\documentclass[12pt]{amsart}
\usepackage{amsmath}
\usepackage{amsthm}
\usepackage{mathrsfs,amssymb,amsfonts,extarrows,textcomp}
\usepackage[all]{xy}
\usepackage{upgreek}
\numberwithin{equation}{section}

\newtheorem{theorem}[subsection]{Theorem}
\newtheorem*{theorem*}{Theorem}
\newtheorem{lemma}[subsection]{Lemma}

\newtheorem{corollary}[subsection]{Corollary}

\newtheorem{proposition}[subsection]{Proposition}

\theoremstyle{definition}
\newtheorem{example}[subsection]{Example}
\newtheorem{definition}[subsection]{Definition}
\newtheorem*{definition*}{Definition}

\theoremstyle{remark}
\newtheorem{remark}[subsection]{Remark}

\makeatletter\@addtoreset{equation}{section}

\makeatother

\renewcommand\labelenumi{(\arabic{enumi})}

\usepackage{hyperref}

\newcommand{\FFF}{\mathbf{F}}
\newcommand{\KKK}{\mathbf{K}}
\newcommand{\LLL}{\mathbf{L}}

\newcommand{\PP}{\mathbb{P}}
\newcommand{\ZZ}{\mathbb{Z}}

\newcommand{\cC}{\mathcal{C}}

\newcommand{\Aut}{\operatorname{Aut}}
\newcommand{\GL}{\operatorname{GL}}
\newcommand{\SL}{\operatorname{SL}}

\newcommand{\PGL}{\operatorname{PGL}}
\newcommand{\PSL}{\operatorname{PSL}}

\newcommand{\rk}{\operatorname{rk}}

\newcommand{\Pic}{\operatorname{Pic}}

\newcommand{\Gal}{\operatorname{Gal}}
\newcommand{\Cr}{\operatorname{Cr}}
\newcommand{\WE}{\mathrm{W}(\mathrm{E}_8)}
\newcommand{\lcm}{\operatorname{lcm}}

\def \ge {\geqslant}
\def \le {\leqslant}

\title{Jordan property for Cremona group over a finite field}

\author{Yuri Prokhorov}
\author{Constantin Shramov}

\address{\emph{Yuri Prokhorov}
\newline
\textnormal{Steklov Mathematical Institute of RAS,
8 Gubkina street, Moscow 119991, Russia.}
\newline
\textnormal{HSE University, Russian Federation,
Laboratory of Algebraic Geometry, 6 Usacheva str., Moscow, 119048, Russia.}
\newline
\textnormal{\texttt{prokhoro@mi-ras.ru}}}

\address{\emph{Constantin Shramov}
\newline
\textnormal{Steklov Mathematical Institute of RAS,
8 Gubkina street, Moscow 119991, Russia.}
\newline
\textnormal{HSE University, Russian Federation,
Laboratory of Algebraic Geometry, 6 Usacheva str., Moscow, 119048, Russia.}
\newline
\textnormal{\texttt{costya.shramov@gmail.com}}}

\thanks{This work is supported by the Russian Science Foundation under grant \textnumero 18-11-00121.}

\date{}

\begin{document}

\begin{abstract}
We show that the Cremona group of rank $2$ over a finite field is Jordan, and provide an upper bound for its Jordan constant
which is sharp when the number of elements in the field is different from~$2$, $4$, and~$8$.
\end{abstract}

\maketitle
\tableofcontents

\section{Introduction}

Groups of geometric origin, such as automorphism and birational automorphism groups of algebraic varieties, often have complicated structure. In particular,
even the finite subgroups therein may be hard to classify explicitly. To provide an approach to a description of such subgroups, V.\,Popov
introduced the following notion, inspired by the classical result of C.\,Jordan on finite subgroups of the general linear group \cite[\S40]{Jordan}
and the later work of J.-P.\,Serre~\cite{Serre-2009}.

\begin{definition}[{\cite[Definition~2.1]{Popov11}}]
\label{definition:Jordan}
A group~$\Gamma$ is called \emph{Jordan}
(alternatively, one says that~$\Gamma$ has \emph{Jordan property}),
if there exists a constant $J=J(\Gamma)$ such that any
finite subgroup of~$\Gamma$
contains a normal abelian subgroup of index at most~$J$.
\end{definition}

One of the most popular groups in birational geometry is the \emph{Cremona group of rank~$2$}, or just the \emph{Cremona group},
$\Cr_2(\KKK)$ over a field $\KKK$, which is the group of birational selfmaps of the projective plane over $\KKK$.
J.-P.\,Serre proved in \mbox{\cite[Theorem~5.3]{Serre-2009}} that
every finite subgroup of $\Cr_2(\KKK)$ of order coprime to the characteristic of $\KKK$
contains a normal abelian subgroup whose index is bounded by some universal constant.
In particular, the Cremona group over a field of characteristic zero is Jordan.

On the other hand, it is easy to see that Jordan property fails for the Cremona group
over an algebraically closed field $\KKK$ of positive characteristic~$p$, and even for the much smaller
group $\PGL_2(\KKK)$. Indeed, the latter group contains arbitrarily large finite simple subgroups
$\PSL_2(\FFF_{p^r})$, where $\FFF_{p^r}$ is the field of~$p^r$ elements.
In general, it is known that for groups of automorphisms and birational automorphisms of algebraic varieties
over a field of characteristic $p>0$, the so-called $p$-Jordan property makes more sense than
the usual Jordan property. We refer the reader
to \cite{BrauerFeit}, \cite{LarsenPink}, \cite{Hu}, \cite{ChenShramov}, and~\cite{Kuznetsova}
for further results in this direction.

However, the above argument for the failure
of Jordan property does not apply for finite fields. The goal of this paper is to show that
the Cremona group of rank $2$ over such a field is indeed Jordan.

Let $p$ be a prime number, let $k$ be a positive integer, and let $\FFF_q$ denote the finite field of $q$ elements, where
$q=p^k$. Set
$$
|\WE|=2^{14}\cdot 3^5\cdot 5^2\cdot 7=696\,729\,600,
$$
so that $|\WE|$ is the order of the Weyl group~$\WE$.
We are going to prove

\begin{theorem}\label{theorem:main}
Every finite subgroup of the Cremona group $\Cr_2(\FFF_q)$ contains a normal abelian
subgroup of index at most $|\WE|$ if $q\in\{2,4,8\}$, and of index at most~\mbox{$q^3(q^2-1)(q^3-1)$}
otherwise; in the latter case, this bound is sharp. In particular, the group $\Cr_2(\FFF_q)$ is Jordan.
\end{theorem}

We will see in Example~\ref{example:Fn} that $\Cr_2(\FFF_q)$ contains finite subgroups whose order is an arbitrarily large power of~$p$.

Using the terminology of \cite[Definition~1]{Popov14}, one may say that Theorem~\ref{theorem:main}
provides an upper bound for the Jordan constant of the group $\Cr_2(\FFF_q)$, and computes this constant
for all $q\not\in\{2,4,8\}$.
It would be interesting to find out the exact values of the Jordan constant of $\Cr_2(\FFF_q)$ for $q\in\{2,4,8\}$
(cf.~Example~\ref{example:cubic}).

It was proved in \cite{Cantat} that for odd $q$
there exists a surjective homomorphism
$$
\Cr_2(\FFF_q)\longrightarrow \mathfrak{S}_{q^2+q+1}.
$$
Together with Theorem~\ref{theorem:main},
this provides a nice geometric example of a group
having smaller Jordan constant than its quotient group.

\medskip
Let us sketch the idea of the proof of Theorem~\ref{theorem:main}.
As usual, studying finite subgroups of $\Cr_2(\FFF_q)$ boils down to
studying finite groups acting biregularly on del Pezzo surfaces and conic bundles over~$\PP^1$.
The former groups do not provide obstructions to the desired Jordan property of
$\Cr_2(\FFF_q)$ due to the boundedness of del Pezzo surfaces (and the fact that
the automorphism group of any such surface over $\FFF_q$ is finite).
As for the groups acting on a conic bundle $\phi\colon S\to \PP^1$, they are mostly contained in the automorphism
group of the scheme-theoretic generic fiber $\cC$ of $\phi$. Note that $\cC$ is a conic over the infinite field~$\FFF_q(t)$, where
$t$ is an independent variable. However, using the fact that $\FFF_q(t)$ is purely transcendental over $\FFF_q$,
one can show that $\Aut(\cC)$ contains few finite subgroups. More precisely, it is possible to provide a uniform bound
for Jordan constants of the groups~$\Aut(\cC)$ for all conics $\cC$ over $\FFF_q(t)$ (see Lemma~\ref{lemma:ADE}), which is enough to deduce
Theorem~\ref{theorem:main}.

The plan of the paper is as follows.
In Section~\ref{section:groups} we collect auxiliary facts about finite groups.
In Section~\ref{section:examples} we present examples of large finite subgroups of Cremona group over finite fields.
In Section~\ref{section:conics} we make observations concerning the automorphism groups of conics over
purely transcendental field extensions of~$\FFF_q$.
In Section~\ref{section:DP} we study automorphism groups of del Pezzo surfaces over finite fields.
In Section~\ref{section:CB} we study automorphism groups of conic bundles over finite fields
and complete the proof of Theorem~\ref{theorem:main}.

We use the following notation and conventions. For a field $\FFF$, by $\bar{\FFF}$
we denote its algebraic closure. If $X$ is a variety over a field $\KKK$, and $\KKK\subset\LLL$
is a field extension, we denote by $X_\LLL$ the extension of scalars of $X$ to $\LLL$.
By a conic we mean a smooth curve of degree $2$ in~$\PP^2$.
A del Pezzo surface is a smooth projective surface $S$ whose anticanonical divisor
$-K_S$ is ample; the degree of~$S$ is the self-intersection number~$K_S^2$.

\medskip
We are grateful to S.\,Gorchinskiy, L.\,Rybnikov, A.\,Trepalin, and V.\,Vologodsky for useful discussions.

\section{Groups}
\label{section:groups}

In this section we collect auxiliary facts about linear groups over finite fields. Most
of them are well known to experts, but we provide their proofs for the reader's convenience.

\begin{lemma}
\label{lemma:GLn-elements}
Let $\tilde{g}\in\GL_n(\FFF_q)$ be an element of order $\tilde{m}$ coprime to $q$. Then
$\tilde{m}\le q^n-1$,
and this bound is sharp.
\end{lemma}

\begin{proof}
Let $\mu\in \FFF_q[x]$ be the minimal polynomial of $\tilde{g}$.
Then $\deg \mu\le n$ and $\mu$ divides~\mbox{$x^{\tilde{m}}-1$}.
Let $\mu=\mu_1\cdot\ldots\cdot \mu_r$ be the decomposition of $\mu$ into a product of
irreducible polynomials. Since $\tilde{m}$ is coprime to $q$, the element $\tilde{g}$ is diagonalizable over $\bar{\FFF}_q$, and hence
the polynomial $\mu$ has no multiple roots.
Thus, the polynomials $\mu_1,\ldots, \mu_r$ are pairwise coprime. Let $\mu_i'=\mu/\mu_i$. Then for some $\nu_i\in \FFF_q[x]$ we can write
\[
\sum \mu_i'\nu_i=1.
\]
Put $h_i= \mu_i'(\tilde{g})\nu_i(\tilde{g})$. Then $h_1,\ldots,h_r$ are orthogonal projectors:
\[
h_ih_j=0\ \text{for}\ i\neq j, \qquad h_i^2=h_i, \qquad \text{and}\qquad \sum h_i=1.
\]
This system of projectors induces a decomposition of
the vector space $V=\FFF_q^n$ into the direct sum
of $\tilde{g}$-invariant subspaces:
\[
V=\bigoplus V_i,\qquad V_i= h_i(V).
\]
Put $\tilde{g}_i= h_i\cdot (\tilde{g}-1)+1$. Then
$\tilde{g}_1,\ldots,\tilde{g}_r$ pairwise commute
and $\tilde{g}=\tilde{g}_1\cdot\ldots\cdot \tilde{g}_r$.
By the construction
\[
\tilde{g}_i(v)=
\begin{cases}
\tilde{g}(v),&\text{if $v\in V_i$,}
\\
v,&\text{if $v\in V_j$, $j\neq i$.}
\end{cases}
\]
In particular, we have $\mu_i(\tilde{g}_i|_{V_i})=0$. Since $\mu_i$ is irreducible, we conclude that
$\mu_i$ is the minimal polynomial of $\tilde{g}_i|_{V_i}$.
Let $\tilde{m}_i$ be the order of $\tilde{g}_i$, let $d_i=\deg \mu_i$,
and let $n_i=\dim V_i$.
Note that $d_i\le n_i$, and $\mu_i$ divides $x^{q^{d_i}-1}-1$ because $\mu_i$ is irreducible.
Hence any eigenvalue of $\tilde{g}_i|_{V_i}$ is a $(q^{d_i}-1)$-th root of unity.
Therefore, $\tilde{m}_i$ divides $q^{d_i}-1$.
Finally, we obtain
\begin{equation}
\label{eq:mm}
\tilde{m}=\lcm(\tilde{m}_1,\ldots \tilde{m}_r)\le \lcm(q^{d_1}-1,\ldots,q^{d_r}-1)\le \prod (q^{n_i}-1)\le q^{n}-1.
\end{equation}

To show that
the equality $m= q^n-1$ is attained, consider the vector space
$$
W=(\FFF_q)^{q^n-1}
$$
and a linear operator $g'\in \GL(W)$ acting
by a cyclic permutation of the basis vectors.
Then the minimal polynomial of $g'$ is
$$
\mu'(x)=x^{q^n-1}-1.
$$
Let~\mbox{$\mu'=\mu'_1\cdot\ldots\cdot \mu'_l$} be the decomposition into a product of
irreducible factors, and let~\mbox{$\theta\in \FFF_{q^n}$} be a primitive $(q^n-1)$-th root of unity.
Then $\theta$ is a root of some factor~$\mu_{i_0}'$.
As above this decomposition induces a decomposition
$$
W=\bigoplus W_i,
$$
and one has $g'=g_1'\cdot\ldots\cdot g_l'$, so that the minimal polynomial of~\mbox{$g'_i|_{W_i}$} is~$\mu'_i$.
Since $\theta$ is a root of $\mu'_{i_0}$, we see that~\mbox{$\deg \mu_{i_0}'=n$}. In particular, this implies $\dim W_{i_0}=n$.
Therefore, the eigenvalues of
$$
\tilde{g}=g_{i_0}'|_{W_{i_0}}\in \GL(W_{i_0})\cong \GL_n(\FFF_q)
$$
are primitive $(q^n-1)$-th roots of unity.
This means that $\tilde{g}$
is an element of order $q^n-1$ in $\GL_n(\FFF_q)$.
\end{proof}

\begin{lemma}
\label{lemma:PGLn-elements}
Let $g\in\PGL_n(\FFF_q)$ be an element of order $m$ coprime to $q$. Then
$$
m\le \frac{q^n-1}{q-1},
$$
and this bound is sharp.
\end{lemma}

\begin{proof}
Consider the natural homomorphism $\pi\colon \GL_n(\FFF_q)\to \PGL_n(\FFF_q)$.
Let $\tilde{g}\in \pi^{-1}(g)$ be a lifting of $\tilde{g}$ to $\GL_n(\FFF_q)$.
Since the kernel of $\pi$ does not contain elements of order $p$,
the order $\tilde{m}$ of $\tilde{g}$ is coprime to $q$.

Let us use the notation from the proof of Lemma~\ref{lemma:GLn-elements}.
First consider the case when~\mbox{$r\ge 2$}, i.e. the minimal polynomial $\mu(x)$ of $\tilde{g}$ is reducible.
Note that every number
$q^{d_i}-1$ is divisible by $q-1$.
Thus by~\eqref{eq:mm} we have
\[
m\le \tilde{m} \le \lcm(q^{d_1}-1,\ldots,q^{d_r}-1)\le \frac{1}{q-1} \prod (q^{d_i}-1)\le \frac {q^{n}-1}{q-1}.
\]

Now consider the case when $r=1$, i.e. $\mu$ is irreducible. Denote $d=d_1$.
Then $\mu$ divides~\mbox{$x^{q^d-1}-1$}, so all the roots $\varepsilon_1,\ldots,\varepsilon_d$ of $\mu$ lie in
$\FFF_{q^d}$. Hence
$$
\lambda_1=\varepsilon_1^{\frac{q^d-1}{q-1}}
$$
is an element of $\FFF_q$, so that $\varepsilon_1$ is a common root of the polynomials $\mu$ and
$$
x^{\frac{q^d-1}{q-1}}-\lambda_1\in \FFF_q[x].
$$
Since $\mu$ is irreducible, it divides
$x^{\frac{q^d-1}{q-1}}-\lambda_1$. This implies that
$\varepsilon_1,\ldots,\varepsilon_d$ are roots of the polynomial~\mbox{$x^{\frac{q^d-1}{q-1}}-\lambda_1$} and so
\[
\varepsilon_1^{\frac{q^d-1}{q-1}}= \ldots= \varepsilon_d^{\frac{q^d-1}{q-1}},
\]
i.e. $\tilde{g}^{\frac{q^d-1}{q-1}}$ is a scalar matrix. Therefore, we have
\begin{equation*}
m\le \frac{q^d-1}{q-1}\le \frac{q^n-1}{q-1}.
\end{equation*}

To show that the above bound for $m$ is attained, consider
the element~\mbox{$\tilde{g}\in \GL_n(\FFF_q)$} of order $q^n-1$, which exists by Lemma~\ref{lemma:GLn-elements}.
Set $g=\pi(\tilde{g})$, and let $m$ be the order of $g$.
Then~$\tilde{g}^m$ is a scalar matrix over~$\FFF_q$, so that
$\tilde{g}^{m(q-1)}=1$. This gives $m(q-1)\ge q^n-1$, and hence
\[
m=\frac{q^n-1}{q-1}.
\qedhere
\]
\end{proof}

\begin{lemma}
\label{lemma:PSL2-elements}
Let $g\in\PSL_2(\FFF_q)$ be an element of order $m$ coprime to $q$. Then $m\le \frac{q+1}{2}$ if $q$ is odd, $m\le q+1$ if $q$ is even,
and these bounds are sharp.
\end{lemma}

\begin{proof}
If $q$ is even, then $\PSL_2(\FFF_q)=\PGL_2(\FFF_q)$. Therefore, by Lemma~\ref{lemma:PGLn-elements} we may assume that $q$ is odd.

Choose a lifting $\tilde{g}$ of the element $g\in\PSL_2(\FFF_q)$ to $\SL_2(\FFF_q)$.
If the minimal polynomial~$\mu$ of $\tilde{g}$ is reducible, then $\tilde{g}$ is conjugate in $\GL_2(\FFF_q)$ to a diagonal matrix.
In this case~\mbox{$\tilde{g}^{q-1}= 1$} and
$$
\tilde{g}^{\frac{q-1}2}=\pm 1,
$$
so that $m\le (q-1)/2$.
Thus we may assume that $\mu$ is irreducible.
The matrix $\tilde{g}$
is conjugate in $\GL_2(\FFF_{q^2})$ to a diagonal matrix. Let $\varepsilon_1, \varepsilon_2\in\FFF_{q^2}\setminus\FFF_q$ be its eigenvalues.
Let~$\sigma$ be the element of order $2$ in the Galois group~\mbox{$\Gal(\FFF_{q^2}/\FFF_{q})\cong\ZZ/2\ZZ$}.
We have
$$
\varepsilon_1^{-1}=\varepsilon_2=\sigma(\varepsilon_1)=\varepsilon_1^{q},
$$
so that $\varepsilon_1^{q+1}=1$. Hence $\varepsilon_1^{\frac{q+1}{2}}=\pm1$,
which implies
$$
\varepsilon_2^{\frac{q+1}{2}}=\varepsilon_1^{-\frac{q+1}{2}}=\varepsilon_1^{\frac{q+1}{2}}.
$$
The latter means that $\tilde{g}^{\frac{q+1}{2}}$ is a scalar matrix, and
the order of $g$ is at most $(q+1)/2$.

To prove that the bound for the order is attained, let $\zeta\in \FFF_{q^2}$ be a primitive root of unity of degree $q^2-1$. Then
$$
\zeta^{q-1}\cdot \sigma(\zeta^{q-1})=\zeta^{q-1}\cdot\zeta^{q(q-1)}=\zeta^{q^2-1}=1,
$$
and $a=\zeta^{q-1}+\sigma(\zeta^{q-1})$ is an element of $\FFF_{q}$.
Set $\mu(x)= x^2-ax+1$.
Then $\zeta^{q-1}$ and~\mbox{$\sigma(\zeta^{q-1})$}
are primitive roots of unity of degree $q+1$
and they are the roots of $\mu$. Hence,
$\mu$ is an irreducible polynomial. Note that $\mu$ is the minimal polynomial of the matrix
$$
\tilde{g}=\begin{pmatrix}0&-1\\ 1& a
\end{pmatrix}\in \SL_2(\FFF_q).
$$
Thus, $\zeta^{q-1}$ and $\sigma(\zeta^{q-1})$ are the eigenvalues of $\tilde{g}$. In particular,
one has $\tilde{g}^{\frac{q+1}2}=-1$.
On the other hand, for $0<r<(q+1)/2$ we have
$$
\big(\zeta^{q-1}\big)^r\neq \big(\zeta^{q-1}\big)^{-r}=\big(\sigma(\zeta^{q-1})\big)^r,
$$
so that $\tilde{g}^r$ is not a scalar matrix.
Hence the order of $\pi(\tilde{g})$ is equal to $(q+1)/2$.
\end{proof}

\begin{lemma}\label{lemma:Wilson}
Let $n\ge 2$ be an integer; if $n=2$, assume that $q\ge 4$.
Then the group~\mbox{$\PGL_n(\FFF_q)$} does not contain non-trivial normal abelian subgroups.
\end{lemma}

\begin{proof}
Recall that $\PSL_n(\FFF_q)$ is a simple non-abelian group
if $n>2$ or $q>3$, see for instance~\mbox{\cite[\S3.3.1]{Wilson}}.
Also, it is straightforward to show that the centralizer of $\PSL_n(\FFF_q)$
in $\PGL_n(\FFF_q)$ is trivial.

Suppose that $\PGL_n(\FFF_q)$ contains a non-trivial normal abelian subgroup $A$.
Since the group~\mbox{$\PSL_n(\FFF_q)$} is simple and non-abelian, we conclude that
the intersection~\mbox{$A\cap \PSL_n(\FFF_q)$} is trivial. Since both $A$ and $\PSL_n(\FFF_q)$
are normal in $\PGL_n(\FFF_q)$, this implies that they commute with each other.
In other words, $A$ is contained in the centralizer of~\mbox{$\PSL_n(\FFF_q)$}, which is a contradiction.
\end{proof}

\section{Examples}
\label{section:examples}

In this section we present examples of large finite subgroups of Cremona group over finite fields.
Recall that if $S$ is a rational surface over $\FFF_q$, then $\Aut(S)$ is isomorphic
to a subgroup of $\Cr_2(\FFF_q)$.

\begin{example}\label{example:P2}
Let $S=\PP^2$. Then $\Aut(S)\cong\PGL_3(\FFF_q)$. Therefore, we have
$$
|\Aut(S)|=q^3(q^2-1)(q^3-1),
$$
and $\Aut(S)$ does not contain non-trivial normal abelian subgroups
by Lemma~\ref{lemma:Wilson}.
\end{example}

\begin{example}\label{example:P1xP1}
Let $S=\PP^1\times\PP^1$. Then
$$
\Aut(S)\cong\big(\PGL_2(\FFF_q)\times\PGL_2(\FFF_q)\big)\rtimes\ZZ/2\ZZ.
$$
Therefore, we have
$$
|\Aut(S)|=2q^2(q^2-1)^2.
$$
If $q\ge 4$, it follows from Lemma~\ref{lemma:Wilson} that
$\Aut(S)$ does not contain non-trivial normal abelian subgroups.
Note however that
$$
2q^2(q^2-1)^2<q^3(q^2-1)(q^3-1)
$$
for every $q$. Thus, the lower bound for the Jordan constant of $\Cr_2(\FFF_q)$
provided by the group~\mbox{$\Aut(S)$} is weaker than that provided by $\Aut(\PP^2)$, in contrast with the
case of algebraically closed field of characteristic~$0$, cf. \cite[Theorem~1.9]{Yasinsky}.
\end{example}

\begin{example}\label{example:dP5}
Let $S$ be a del Pezzo surface of degree $5$ obtained as a blow up of four points
in general position on $\PP^2$ over $\FFF_q$, i.e.
such points that no three of them are collinear.
Note that such a quadruple of points exists for any~$q$, because for three lines in~$\PP^2$ not passing through one point
there always exists a point outside these lines.
One has~\mbox{$\Aut(S)\cong\mathfrak{S}_5$}, so that~\mbox{$|\Aut(S)|=120$}, and $\Aut(S)$ does not
contain non-trivial normal abelian subgroups. Since the surface $S$ is rational, the group~\mbox{$\Aut(S)$} is realized as a subgroup of~$\Cr_2(\FFF_q)$. However, the lower bound for the Jordan constant of $\Cr_2(\FFF_q)$
provided by~\mbox{$\Aut(S)$} is always weaker than that provided by~\mbox{$\Aut(\PP^2)$} and $\Aut(\PP^1\times\PP^1)$.
\end{example}

\begin{example}\label{example:cubic}
Let $S$ be the Fermat cubic in $\PP^3$ over $\FFF_4$.
This is the surface given by the equation
\[
x_0^3+x_1^3+x_2^3+x_3^3=0.
\]
One has
$$
\Aut(S)\cong\mathrm{PSU}_4(\FFF_2),
$$
see for instance \cite[\S5.1]{DD}. Thus, $\Aut(S)$ is a non-abelian simple group of order~$25\,920$.
Note that there are two skew lines $L_1$ and $L_2$ on $S$ defined by equations
\[
x_0+\omega x_1= x_2+\omega x_3= 0\quad\text{and}\quad x_0+\omega^2 x_1= x_2+\omega^2 x_3=0,
\]
respectively,
where $\omega\in \FFF_4$ is a non-trivial cubic root of unity.
Hence $S$ is rational, and so the group~\mbox{$\Aut(S)$} is realized as a subgroup of~$\Cr_2(\FFF_4)$. We point out that the lower bound for the Jordan constant of $\Cr_2(\FFF_4)$
provided by~\mbox{$\Aut(S)$} is still weaker than that provided by~\mbox{$\Aut(\PP^2)$}.
\end{example}

\begin{example}\label{example:Fn}
Let $S=\PP(1,1,n)$ be a weighted projective plane with weighted homogeneous
coordinates $x_0$, $x_1$, and $x_2$ of weights $1$, $1$, and $n$, respectively.
Then the group~$\Aut(S)$ contains a subgroup $\Gamma$ that consists
of coordinate changes
$$
(x_0:x_1:x_2)\mapsto (x_0:x_1:x_2+P(x_0,x_1)),
$$
where $P$ is a homogeneous polynomial of degree $n$ in two variables.
Thus, $\Gamma$ is isomorphic to the group of points of the vector space of dimension $n+1$ over $\FFF_q$.
Therefore, the group $\Cr_2(\FFF_q)$ contains finite subgroups whose order is an arbitrarily large power of~$p$.
\end{example}

\section{Conics over non-perfect field}
\label{section:conics}

In this section we make several observations concerning the automorphism groups of conics over purely transcendental field extensions of~$\FFF_q$.
We start with a simple algebraic fact.

\begin{lemma}\label{lemma:finite-extension}
Let $\KKK\supset \FFF_q$ be a purely transcendental field extension, and let $\LLL\supset \KKK$ be a field extension of finite degree $l$.
Suppose that the multiplicative group $\LLL^*$ contains an element $\zeta$ of finite order~$m$.
Then $m$ divides~\mbox{$q^l-1$}.
\end{lemma}

\begin{proof}
The algebraic closure of $\FFF_q$ in $\LLL$ is a finite extension of $\FFF_q$; denote it by $\FFF_{q'}$.
Since~\mbox{$\FFF_q\subset \FFF_{q'}$} is a separable algebraic extension, and $\FFF_q$ is separably closed in $\KKK$,
the ring
$$
\LLL'=\KKK\otimes_{\FFF_q}\FFF_{q'}
$$
is a field, see \cite[Theorem~IV.21(2)]{Jacobson}.
We have natural embeddings
$$
\KKK\hookrightarrow \LLL'\hookrightarrow\LLL.
$$
Denote by $d$ the degree of the field extension $\KKK\subset\LLL'$.
Then $d$ divides $l$, and since~\mbox{$\KKK\cap\FFF_{q'}=\FFF_q$}, we see that $d$ equals
the degree of the field extension~\mbox{$\FFF_q\subset\FFF_{q'}$}.
On the other hand, the element $\zeta\in\LLL$ is contained in $\FFF_{q'}$, because it is algebraic over $\FFF_q$.
Hence its order in $\LLL^*$ divides $|\FFF_{q'}^*|=q^d-1$, which in turn divides~\mbox{$q^l-1$}.
\end{proof}

\begin{remark}
If $l=1$, then by Lemma~\ref{lemma:finite-extension} for every element of finite order in $\KKK^*$,
its order divides $q-1$. In other words, the group $\KKK^*$ contains an element of finite order~$m$
if and only if the group $\FFF_q^*$ contains an element of the same order.
\end{remark}

Lemma~\ref{lemma:finite-extension} allows one to bound the finite orders of the elements in the
automorphism groups of conics over purely transcendental field extensions of~$\FFF_q$.

\begin{lemma}\label{lemma:PGL2-Fqt}
Let $\KKK\supset \FFF_q$ be a purely transcendental field extension, and let
$\cC$ be a conic over~$\KKK$.
Suppose that the group $\Aut(\cC)$ contains an element of finite order $m$ coprime to~$q$.
Then $m$ divides~\mbox{$q^2-1$}.
\end{lemma}

\begin{proof}
Let $g\in\Aut(\cC)$ be an element of order $m$. Since $m$ is coprime to $q$, we conclude that $g$
is a semi-simple element in
$$
\Aut\big(\cC_{\bar{\KKK}}\big)\cong\PGL_2\left(\bar{\KKK}\right).
$$
Hence $g$ is contained in a torus in $\PGL_2(\bar{\KKK})$, and so it has exactly two fixed points on
$$
\cC_{\bar{\KKK}}\cong\PP^1_{\bar{\KKK}}.
$$
Thus, there exists a degree $2$ field extension $\LLL\supset \KKK$ such that $g$ has a fixed point~$P$ on~$\cC_{\LLL}$.
Therefore, $g$ acts by an automorphism of order $m$ on the Zariski
tangent space~\mbox{$T_P(\cC_{\LLL})\cong \LLL$}, see for instance~\mbox{\cite[Theorem~3.7]{ChenShramov}}.
So $\GL\big(T_P(\cC_{\LLL})\big)\cong\LLL^*$ contains an element of order $m$.
Now the assertion follows from Lemma~\ref{lemma:finite-extension}.
\end{proof}

Next, we study finite subgroups of certain particular types
in automorphism groups of conics over purely transcendental field extensions of~$\FFF_q$;
recall that $q=p^k$.

\begin{lemma}\label{lemma:PSL}
Let $\KKK\supset \FFF_q$ be a purely transcendental field extension, and let $\cC$ be a conic over~$\KKK$.
Then the group $\Aut(\cC)$ does not contain
subgroups isomorphic to $\PSL_2(\FFF_{p^r})$ for~\mbox{$p^r>q$}.
\end{lemma}

\begin{proof}
By Lemma~\ref{lemma:PSL2-elements},
the group $\PSL_2(\FFF_{p^r})$ contains an element whose order equals~\mbox{$\frac{p^r+1}{2}$}
if $p$ is odd, and an element whose order equals
$2^r+1$ if $p=2$.
On the other hand, by Lemma~\ref{lemma:PGL2-Fqt} the finite orders of elements
of the group $\Aut(\cC)$ that are coprime to $p$ divide~\mbox{$q^2-1=p^{2k}-1$}; in particular, we have $2k\ge r$.
If $p$ is odd and $r>k$, then
$$
2p^{2k-r}+1<3p^{2k-r}\le p^{2k-r+1}<p^r,
$$
so that
$$
p^{2k}-1<2p^{2k-r}\cdot\frac{p^r+1}{2}=p^{2k}+p^{2k-r}<\left(p^{2k}-1\right)+\frac{p^r+1}{2}.
$$
In other words, $p^{2k}-1$ does not divide $p^r+1$ in this case.
Similarly, if $r>k$, then
$$
2^{2k}-1<2^{2k-r}(2^r+1)<(2^{2k}-1)+(2^r+1),
$$
so that $2^{2k}-1$ does not divide $2^r+1$.
\end{proof}

We will need the following classification of finite subgroups in the automorphism group
of a projective line.

\begin{theorem}[{see e.g. \cite[Theorem~2.1]{DD}}]
\label{theorem:ADE}
Let $\LLL$ be a field of characteristic~$p$, and let~$G$ be a finite subgroup of $\PGL_2(\LLL)$.
Then $G$ is isomorphic to one of the following groups:
\begin{enumerate}
\renewcommand\labelenumi{\rm (\arabic{enumi})}
\renewcommand\theenumi{\rm (\arabic{enumi})}
\item
\label{theorem:ADE-1}
a dihedral group of order $2m$, $m\ge 2$;

\item
\label{theorem:ADE-2}
one of the groups $\mathfrak{A}_4$, $\mathfrak{S}_4$, or
$\mathfrak{A}_5$;

\item
\label{theorem:ADE-3}
the group $\PSL_2(\mathbf{F}_{p^r})$ for some $r\ge 1$;

\item
\label{theorem:ADE-4}
the group $\PGL_2(\mathbf{F}_{p^r})$ for some $r\ge 1$;

\item
\label{theorem:ADE-5}
a group of the form $G_p\rtimes \ZZ/m\ZZ$, where $m\ge 1$
is coprime to~$p$, and $G_p$ is a subgroup of the additive group of~$\LLL$.
\end{enumerate}
\end{theorem}

Recall that a subgroup of a group $G$ is called \emph{characteristic}
if it is preserved by all automorphisms of~$G$. Using Theorem~\ref{theorem:ADE}, we prove

\begin{lemma}\label{lemma:ADE}
Let $\KKK\supset \FFF_q$ be a purely transcendental field extension, let $\cC$ be a conic over~$\KKK$,
and let $G$ be a finite subgroup of~\mbox{$\Aut(\cC)$}. Then
$G$ contains a characteristic abelian subgroup of index at most
$$
J=\max\{q(q^2-1),\, 60\}.
$$
\end{lemma}

\begin{proof}
We go through the list of possible finite subgroups of
$$
\Aut(\cC)\subset\Aut\big(\cC_{\bar{\KKK}}\big)\cong\PGL_2\left(\bar{\KKK}\right)
$$
provided by Theorem~\ref{theorem:ADE}.
In most cases it will be enough to take the order of the group~$G$
as an upper bound for the smallest index of a characteristic abelian
subgroup therein.

If $G$ is of type~\ref{theorem:ADE-1}, then either $|G|=4$, and~$G$ is abelian itself, or $|G|>4$, and $G$ contains a characteristic cyclic subgroup of index~$2$.
If $G$ is of type~\ref{theorem:ADE-2}, then $|G|\le 60$.
If~$G$ is of type~\ref{theorem:ADE-3} or~\ref{theorem:ADE-4}, then
$$
|G|\le |\PGL_2(\FFF_q)|=q(q^2-1)
$$
by Lemma~\ref{lemma:PSL}.
If $G$ is of type~\ref{theorem:ADE-5}, then there is a unique $p$-Sylow subgroup $G_p$ in $G$, which is therefore characteristic;
the index of $G_p$ in $G$ is at most $q^2-1$ by Lemma~\ref{lemma:PGL2-Fqt}.
Taking the maximum of the above bounds, we obtain the required value of~$J$.
\end{proof}

As a by-product of Lemma~\ref{lemma:ADE}, we obtain

\begin{corollary}\label{corollary:PGL2-Fqt-Jordan}
Let $\KKK\supset \FFF_q$ be a purely transcendental field extension. Then
the group~\mbox{$\PGL_2(\KKK)$} is Jordan.
\end{corollary}

\section{Del Pezzo surfaces}
\label{section:DP}

In this section we study automorphism groups of del Pezzo surfaces over finite fields.
Let us start with analyzing several particular cases.

\begin{lemma}\label{lemma:quadric}
Let $S$ be a surface over $\FFF_q$ such that
$S_{\bar{\FFF}_q}\cong\PP^1\times\PP^1$.
Then
$$
|\Aut(S)|\le 2q^2(q^4-1).
$$
\end{lemma}

\begin{proof}
We use the classification of del Pezzo surfaces
of degree $8$ over arbitrary fields, see for instance~\mbox{\cite[Lemma~3.4(i)]{SV}}.
Namely, if $S$ is a surface over a field $\KKK$  such that~\mbox{$S_{\bar{\KKK}}\cong\PP^1\times\PP^1$},
then either $\rk\Pic(S)=2$ and $S$ is a product of two conics,
or~\mbox{$\rk\Pic(S)=1$} and $S$ is a Weil restriction of scalars of a conic defined over some quadratic
extension of~$\KKK$.

Suppose that $\rk\Pic(S)=2$. Then $S\cong C_1\times C_2$, where $C_1$ and $C_2$ are conics over $\FFF_q$.
Recall that every conic over a finite field has a point. Thus, we have $S\cong\PP^1\times\PP^1$, so that
$$
|\Aut(S)|=2q^2(q^2-1)^2<2q^2(q^4-1).
$$

Now suppose that $\rk\Pic(S)=1$. Then $S$ is isomorphic to the Weil restriction of scalars
$$
S\cong \operatorname{R}_{\FFF_{q^2}/\FFF_q}Q,
$$
where $Q$ is a conic over $\FFF_{q^2}$.
As before, $Q$ has a point over $\FFF_{q^2}$, so that~\mbox{$Q\cong\PP^1_{\FFF_{q^2}}$}. Thus, we have
$$
\Aut(S)\cong \Aut(Q)\rtimes\ZZ/2\ZZ\cong \PGL_2(\FFF_{q^2})\rtimes\ZZ/2\ZZ,
$$
see \cite[Lemma~3.4(iii),(iv)]{SV}.
Therefore, we compute
$$
|\Aut(S)|=2q^2(q^4-1).
$$
\end{proof}

\begin{lemma}\label{lemma:dP2}
Let $S$ be a del Pezzo surface of degree $2$ over $\FFF_q$.
Then $\Aut(S)$ contains a normal abelian subgroup
of index at most $q^3(q^2-1)(q^3-1)$.
\end{lemma}

\begin{proof}
The morphism defined by the anticanonical linear system $|-K_S|$ is a double cover~\mbox{$\kappa\colon S\to \PP^2$}.
The kernel $\Delta$ of the natural homomorphism
$\theta\colon \Aut(S)\to \Aut(\PP^2)$ is a normal subgroup of $\Aut(S)$.
Note that $\Delta$ either is isomorphic to $\ZZ/2\ZZ$ or is trivial, depending on
whether the morphism $\kappa$ is separable or not.
The index of $\Delta$ equals
\[
|\theta(\Aut(S))|\le
|\Aut(\PP^2)|=q^3(q^2-1)(q^3-1).
\qedhere
\]
\end{proof}

\begin{corollary}\label{corollary:dP-intermediate}
Let $S$ be a del Pezzo surface of degree $K_S^2$ different from~$1$ and~$3$ over~$\FFF_q$.
Then~$\Aut(S)$ contains a normal abelian subgroup of index at most~\mbox{$q^3(q^2-1)(q^3-1)$}.
\end{corollary}

\begin{proof}
Note that
$$
q^3(q^2-1)(q^3-1)\ge 2^3(2^2-1)(2^3-1)=168.
$$
Let $d=K_S^2$ be the degree of $S$. It is well known that $1\le d\le 9$.

Suppose that $d=9$, so that $S$ is a Severi--Brauer surface.
Since $\FFF_q$ is a $\mathrm{C}_1$-field by the Chevalley--Warning theorem, see e.g. \cite[\S\,I.2]{Serre1973},
the surface $S$ has an $\FFF_q$-point. Thus~$S$ is isomorphic to $\PP^2$, see~\mbox{\cite[Corollary~13]{Kollar-SB}}.
Hence
$$
|\Aut(S)|=|\PGL_3(\FFF_q)|=q^3(q^2-1)(q^3-1).
$$

Suppose that $d=8$ and $S_{\bar{\FFF}_q}\cong\PP^1\times\PP^1$.
Then
$$
|\Aut(S)|\le 2q^2(q^4-1),
$$
see Lemma~\ref{lemma:quadric}.
Note that
$$
2q^2(q^4-1)<q^3(q^2-1)(q^3-1)
$$
for any $q$.

Suppose that either $d=8$ and $S_{\bar{\FFF}_q}\not\cong\PP^1\times\PP^1$, or $d=7$.
Then there is an $\Aut(S)$-equivariant birational morphism $S\to\PP^2$.
This implies that the bound obtained in the case $d=9$
applies to this case as well.

Suppose that $d=6$. Then
$$
\Aut\big(S_{\bar{\FFF}_q}\big)\cong \left(\bar{\FFF}_q^*\right)^2\rtimes \left(\mathfrak{S}_3\times\ZZ/2\ZZ\right),
$$
see for instance~\mbox{\cite[Theorem~8.4.2]{Dolgachev}}.
This implies that $\Aut(S_{\bar{\FFF}_q})$, and thus also~\mbox{$\Aut(S)$}, contains a normal abelian subgroup of index
at most
$$
|\mathfrak{S}_3\times\ZZ/2\ZZ|=12.
$$

Suppose that $d=5$. Then $\Aut(S_{\bar{\FFF}_q})\cong\mathfrak{S}_5$, see~\mbox{\cite[Theorem~8.5.8]{Dolgachev}}.
Thus, we have
$$
|\Aut(S)|\le |\mathfrak{S}_5|=120.
$$

Suppose that $d=4$. Then $\Aut(S)$ contains a normal abelian subgroup of index at most~$60$, see \cite[Theorem~3.1]{DD}.

Finally, suppose that $d=2$. Then $\Aut(S)$ contains a normal abelian subgroup
of index at most $q^3(q^2-1)(q^3-1)$ by Lemma~\ref{lemma:dP2}.
\end{proof}

For the next case, we assume that the characteristic of the base field is odd.

\begin{lemma}\label{lemma:dP1}
Let $S$ be a del Pezzo surface of degree $1$ over $\FFF_q$.
Suppose that $q$ is odd. Then $\Aut(S)$ contains a normal abelian subgroup
of index at most $4096$.
\end{lemma}

\begin{proof}
The morphism defined by the linear system $|-2K_S|$ is a double cover~\mbox{$\kappa\colon S\to Y$}, where $Y\subset \PP^3$ is a quadratic cone.
The kernel $\Delta$ of the natural homomorphism~\mbox{$\Aut(S)\to \Aut(Y)$} is a central subgroup of $\Aut(S)$; in particular, it is normal and abelian.
Since $q$ is odd, the morphism $\kappa$ is separable.
Therefore, its branch locus consists of the singular
point of $Y$ and a smooth curve~$C$ cut out on $Y$ by a cubic surface, see for instance~\mbox{\cite[Theorem~III.3.5.1]{Kollar-Curves}}.
Thus, $C$ is not contained in a plane in $\PP^3$, and so the action of
the group~\mbox{$\Gamma=\Aut(S)/\Delta$} on $C$ is faithful. On the other hand, the genus of $C$ equals $4$.
In particular, it cannot be represented in the form~\mbox{$\frac{1}{2}p^n(p^n-1)$}.
Hence, according to \cite{St73} one has~\mbox{$|\Aut(C)|\le 16\cdot 4^4=4096$}.
Thus, the group $\Aut(S)$ contains a central subgroup $\Delta$ of index~\mbox{$|\Gamma|\le |\Aut(C)|\le 4096$}.
\end{proof}

\begin{remark}
If $q$ is odd, then one can use the results of \cite{St73}
to obtain an alternative upper bound for the Jordan constant of the group $\Aut(S)$,
where $S$ is a del Pezzo surface of degree $2$ over $\FFF_q$. Namely, in this case the
anticanonical double cover $S\to \PP^2$ is separable, and its branch divisor is a smooth curve $C$ of genus~$3$.
As in the proof of Lemma~\ref{lemma:dP1}, the Jordan constant of $\Aut(S)$ does not exceed the order of the group $\Aut(C)$.
By \cite{St73}, one has $|\Aut(C)|\le 6048$. Note that this upper bound for the Jordan constant of $\Aut(S)$ is weaker than
the one provided by Lemma~\ref{lemma:dP2} if $q=3$, and is stronger for all other odd~$q$.
\end{remark}

The next assertion summarizes the information about Jordan constants of automorphism groups
of del Pezzo surfaces over finite fields.

\begin{proposition}\label{proposition:dP}
Let $S$ be a del Pezzo surface over $\FFF_q$.
Then $\Aut(S)$ contains a normal abelian subgroup of index at most
$$
J_{\mathrm{dP}}=\left\{
\begin{array}{ll}
q^3(q^2-1)(q^3-1),& \text{\ if $q$ is odd},\\
\max\{q^3(q^2-1)(q^3-1), |\WE|\},& \text{\ if $q$ is even}.
\end{array}
\right.
$$
\end{proposition}

\begin{proof}
Let $d=K_S^2$ be the degree of $S$.
If~\mbox{$d\not\in\{1, 3\}$}, then~$\Aut(S)$ contains a normal abelian subgroup of index at most~\mbox{$q^3(q^2-1)(q^3-1)$}
by Corollary~\ref{corollary:dP-intermediate}.
If $d=3$ and $q$ is odd, then~\mbox{$|\Aut(S)|\le 648$}, see \cite[Theorem~1.1]{DD}.
If $d=1$ and $q$ is odd, then $\Aut(S)$ contains a normal abelian subgroup
of index at most $4096$, see Lemma~\ref{lemma:dP1}.

It remains to take care of the cases $d=1$ and $d=3$ when $q$ is even.
In these cases (and actually for arbitrary $q$ and arbitrary $d\le 5$)
the group $\Aut(S)$ is isomorphic to a subgroup of the Weyl group $\WE$,
see for instance~\mbox{\cite[Corollary~8.2.40]{Dolgachev}}.
In particular, one has~\mbox{$|\Aut(S)|\le |\WE|$}.

Overall, we see that $\Aut(S)$ always contains a normal abelian subgroup of index at most
$$
J_{\mathrm{dP}}=\max\{q^3(q^2-1)(q^3-1), 4096\}=q^3(q^2-1)(q^3-1)
$$
if $q$ is odd,
and
$$
J_{\mathrm{dP}}=\max\{q^3(q^2-1)(q^3-1), |\WE|\}
$$
if $q$ is even.
\end{proof}

\section{Conic bundles}
\label{section:CB}

In this section we study automorphism groups of conic bundles over finite fields using the results of Section~\ref{section:conics},
and complete the proof of Theorem~\ref{theorem:main}.
We remind the reader that a conic bundle $\phi\colon S\to C$ is a proper morphism of a smooth geometrically irreducible surface
to a curve such that the anticanonical divisor of $S$ has positive intersection
index with every curve contained in a fiber of~$\phi$, and~\mbox{$\phi_*\mathcal{O}_X=\mathcal{O}_C$}.
The scheme-theoretic generic fiber of a conic bundle $\phi$ is a smooth conic over the field of rational functions on~$C$.

\begin{lemma}\label{lemma:CB}
Let $\phi\colon S\to \PP^1$ be a conic bundle over $\FFF_q$.
Denote by $\Aut(S;\phi)$ the group that consists of all automorphisms of $S$
mapping the fibers of $\phi$ again to the fibers of $\phi$. Then
$\Aut(S;\phi)$ contains a normal abelian subgroup of index at most
$$
J_{\mathrm{CB}}=q(q^2-1)\max\{q(q^2-1), 60\}.
$$
\end{lemma}

\begin{proof}
The group $\Aut(S;\phi)$ fits into the exact sequence
$$
1\longrightarrow \Aut_\phi(S)\longrightarrow \Aut(S;\phi)\longrightarrow \Aut(\PP^1),
$$
where $\Aut_\phi(S)$ is the subgroup that consists of all elements of $\Aut(S;\phi)$
mapping every fiber of $\phi$ to itself.
The index of $\Aut_\phi(S)$ in $\Aut(S;\phi)$ does not exceed
$$
|\Aut(\PP^1)|=|\PGL_2(\FFF_q)|=q(q^2-1).
$$
On the other hand, $\Aut_\phi(S)$ is a subgroup of the automoprhism group
of the scheme-theoretic generic fiber $\cC$ of $\phi$. Observe that $\cC$ is a conic over the field $\FFF_q(\PP^1)\cong\FFF_q(t)$, where $t$ is an independent variable.
Hence, by Lemma~\ref{lemma:ADE} the group $\Aut_\phi(S)$ contains a characteristic abelian subgroup $H$ of index at most
$$
J=\max\{q(q^2-1), 60\}.
$$
Therefore, $H$ is a normal abelian subgroup of $\Aut(S;\phi)$, and its index in $\Aut(S;\phi)$
does not exceed $J_{\mathrm{CB}}=q(q^2-1)\cdot J$.
\end{proof}

Now we are ready to prove Theorem~\ref{theorem:main}.

\begin{proof}[Proof of Theorem~\textup{\ref{theorem:main}}]
Let $G$ be a finite subgroup of $\Cr_2(\FFF_q)$.
Regularizing the birational action of~$G$ on $\PP^2$ and
taking a $G$-equivariant resolution
of singularities, we obtain a smooth projective rational surface $S'$ over $\FFF_q$ with an action of $G$,
see e.g. \cite[Lemma~14.1.1]{P:G-MMP} or~\mbox{\cite[Lemma~3.6]{ChenShramov}}.
Running a $G$-Minimal Model Program on $S'$, we arrive to a rational $G$-minimal surface $S$.
Thus, $S$ is either a del Pezzo surface, or it has a structure of a $G$-equivariant conic bundle, see~\mbox{\cite[Theorem~1G]{Iskovskikh80}
or~\cite[Theorem~2.7]{Mori} (cf. also~\mbox{\cite[Corollary~7.3]{Badescu}}).}
Since~$S$ is rational, it has an $\FFF_q$-point,
see for instance~\mbox{\cite[Lemma~1.1]{VA}}.
If~$S$ is a del Pezzo surface, then we know from Proposition~\ref{proposition:dP} that $G$ contains a normal abelian subgroup of index at most
$$
J_{\mathrm{dP}}=\left\{
\begin{array}{ll}
q^3(q^2-1)(q^3-1),& \text{\ if $q$ is odd},\\
\max\{q^3(q^2-1)(q^3-1), |\WE|\},& \text{\ if $q$ is even}.
\end{array}
\right.
$$

Suppose that $S$ has a structure $\phi\colon S\to C$ of a $G$-equivariant conic bundle. Then $C$ is a smooth geometrically rational curve with an $\FFF_q$-point, which means
that $C\cong\PP^1$. Therefore, we know from Lemma~\ref{lemma:CB}
that $G$ contains a normal abelian subgroup of index at most
$$
J_{\mathrm{CB}}=q(q^2-1)\max\{q(q^2-1), 60\}.
$$

It remains to notice that
$$
\max\{J_{\mathrm{dP}},J_{\mathrm{CB}}\}=J_{\mathrm{dP}}.
$$
Note that one has $J_{\mathrm{dP}}=q^3(q^2-1)(q^3-1)$ unless $q\in\{2,4,8\}$. Moreover, for~\mbox{$q\not\in\{2,4,8\}$} this bound for the indices of normal abelian
subgroups in finite subgroups of $\Cr_2(\FFF_q)$ is sharp by Example~\ref{example:P2}.
\end{proof}

\end{document}